\documentclass[10pt,
]{amsart}
\usepackage{
amssymb
}
\newcommand{\no}[1]{#1}

\renewcommand{\no}[1]{}  \newcommand{\upDelta}{\Delta} 

\no{\usepackage{times}\usepackage[
subscriptcorrection, slantedGreek, 
nofontinfo]{mtpro}
\renewcommand{\Delta}{\upDelta}
}

\date{\today}

\newtheorem{theorem}{Theorem}[section]
\newtheorem{proposition}{Proposition}[section]
\newtheorem{lemma}{Lemma}[section]

\theoremstyle{remark}
\newtheorem{example}{Example}

\theoremstyle{definition}
\newtheorem{remark}{Remark}[section]

\DeclareMathOperator{\Vol}{Vol}
\DeclareMathOperator{\dist}{dist}

\DeclareMathOperator{\supp}{supp}

\DeclareMathOperator{\II}{II}

\newcommand{\eps}{\varepsilon}

\newcommand{\R}{{\bf R}}
\newcommand{\Id}{\mbox{Id}}
\renewcommand{\r}[1]{(\ref{#1})}
\newcommand{\PDO}{$\Psi$DO}
\newcommand{\be}[1]{\begin{equation}\label{#1}}
\newcommand{\ee}{\end{equation}}

\renewcommand{\d}{\mathrm{d}}

\newcommand{\bo}{\partial \Omega}

\title[Recovery of a source or a speed]{Recovery of a source term or a speed with one measurement and applications}

\author[P. Stefanov]{Plamen Stefanov}
\address{Department of Mathematics, Purdue University, West Lafayette, IN 47907}
\thanks{First author partly supported by a NSF FRG Grant DMS-0800428 and a Simons Visiting Professorship}

\author[G. Uhlmann]{Gunther Uhlmann}
\address{Department of Mathematics, University of Washington, Seattle, WA 98195, and UC Irvine, CA 92697}
\thanks{Second author partly supported by a NSF FRG grant No.~0554571 and a Senior Clay Award and Chancellor Professorship at UC Berkeley}

\usepackage{graphicx}
\graphicspath{%
    {converted_graphics/}
    {/}
}
\begin{document}
\maketitle
\begin{abstract}
We study the problem of recovery the source $a(t,x)F(x)$ in the wave equation in anisotropic medium with $a$ known so that $a(0,x)\not=0$ with a single measurement. We use Carleman estimates combined with geometric arguments and give sharp conditions for uniqueness. We also study the non-linear problem of recovery the sound speed in the equation $u_{tt} -c^2(x)\Delta u =0$ with one measurement. We give sharp conditions for stability, as well. An application to thermoacoustic tomography is also presented.  
\end{abstract}

\section{Introduction} 
The purpose of this paper is to give sharp conditions for a recovery of a source term in the wave equation in anisotropic media modeled by a Riemannian metric by a single boundary measurement. In the process we give a more geometric treatment of the problem. This linear problem appears as a linearization and actually, as the full non-linear version, of the problem of recovery a sound speed, given the source. It has applications to thermoacoustic tomography. We are also inspired by the related works \cite{ImanuvilovY01,ImanuvilovY01A, ImanuvilovY03},   \cite[Theorem~8.2.2]{Isakov-book}. The method in these papers uses Carleman estimates for hyperbolic inverse problems that originates in the work of \cite{Bukh_Kl_81} who considered the case of the wave equation with a potential with non-zero initial data.

The main problem that we have in mind is the following. Let $c(x)>0$ and let $u$ solve
\begin{equation}   \label{1A}
\left\{
\begin{array}{rcll}
(\partial_t^2 -c^2\Delta)u &=&0 &  \mbox{in $(0,T)\times \R^n$},\\
u|_{t=0} &=& f,\\ \quad \partial_t u|_{t=0}& =&0, 
\end{array}
\right.               
\end{equation}
where $c=c(x)>0$ and $T>0$ is fixed. Let $c=1$ outside some domain $\Omega$ with a smooth strictly convex  boundary. Given $f$, and  $u$ restricted to $[0,T]\times\bo$, is it possible to reconstruct the speed $c$? Ideally, we want to do this with data on a part of $\bo$, as well. Next, assuming that we can, how stable is this? Clearly, some conditions on $f$ are needed since when $f=0$, for example, we get no information about $c$. This inverse problem is clearly non-linear. 

If we have two speeds $c$ and $\tilde c$, then $w=\tilde u-u$ solves
\begin{equation}   \label{w2}
\left\{
\begin{array}{rcll}
(\partial_t^2 -c^2\Delta)w &=&a(t,x)F(x)&  \mbox{in $(0,T)\times \R^n$},\\
w|_{t=0} &=& 0,\\ \quad \partial_t w|_{t=0}& =&0,
\end{array}
\right.               
\end{equation}
with  
\be{Fa}
F := \tilde c^2-c^2,\quad    a=\Delta\tilde u.
\ee
We consider the more general linear problem of recovery of a function  $F$, given $a$ and $w$ restricted to $[0,T]\times\bo$, or on a part of it. Again, some condition on $a$ is needed since when $a=0$ for example, $F$ cannot be recovered. 
  We actually replace $c^2\Delta$ in \r{w2} by the Laplace-Beltrami operator $\Delta_g$ related to some metric, plus lower order terms. 

Similar problems but for the recovery of a potential or the term $p(x)$ in $\nabla\cdot p\nabla$ are studied in \cite{ImanuvilovY01,ImanuvilovY01A,ImanuvilovY03}. A general abstract theorem of this type is presented in \cite[Theorem~8.2.2]{Isakov-book}. In \cite{ImanuvilovY03} and \cite{Isakov-book}, the principal part of the wave equation has variable coefficients, thus the geometry is non-Euclidean. This requires  some assumptions on the speed or the metric. The method of the proof is to use Carleman estimates, and the assumptions are needed to satisfy the pseudo-convexity condition. Those assumptions are not sharp however. In fact, the proofs are essentially ``Euclidean'', and roughly speaking, the conditions on the speed or on the metric require that the proof still works under an Euclidean treatment. Also, one global pseudo-convex function is used. 
One of the goals of this work is to sharpen those conditions thus extending the results to a larger class of speeds/metrics, formulate them in a geometric way and also prove local results.

There are many works on related problems, including boundary control, for example,  \cite{BardosLR_control, LasieckaTY, Yao_control1,TriggianiY}. The conditions on the metric there are more geometric, requiring existence of a global convex function, or a somewhat general condition of existence of a global vector filed with certain properties. The proofs are still based on Carleman estimates but the goal is to recover non-trivial initial conditions, assuming, say, a Neumann boundary condition, and measuring Dirichlet data on a part of $\bo$. The conditions on $T$ are formulated in terms of lower bounds of the speed and are not sharp.  The geometry of the rays in those problems however is different than the application that we have in mind --- there are reflections at the boundary. On the other hand, the methods there could probably be adapted to the problems studied in the works that we cited above. 

The pseudo-convexity condition needed for the Carleman estimates that we use is satisfied by assuming that the region where we prove unique continuation is foliated by a continuous family $\Sigma_s$ of strictly convex surfaces. In case of data on a part $\Gamma$ of $\bo$, we require those surfaces to intersect $\bo$ in $\Gamma$. In contrast to the other works in this direction, we are not trying to construct one strongly pseudo-convex function. Instead, we prove unique continuation by incremental steps, each time  using a different strongly pseudo-convex function. 

We describe the results in the paper now. 
We start in section~\ref{sec_un1} with the uniqueness Theorem~\ref{thm_uq1}, that is a version of \cite[Theorem~8.2.2]{Isakov-book}. The time interval is $(-T,T)$, there is no initial condition for $w_t$ at $t=0$, and we study the problem of unique recovery of $F$ in \r{w2} given Cauchy data on a part of $(-T,T)\times\bo$. We view this theorem more a as a tool than a goal, and the requirement that the time interval is symmetric about $t=0$ will be satisfied later by studying problems with solutions that have even extensions in $t$, like \r{1A}. In the rest of that section we show two ways  to satisfy the convexity requirement. First, assuming  that $\bo$ is strictly convex, we  show in Theorem~\ref{thm_2} that $F=0$ in some collar neighborhood of $\bo$ of the kind $\dist(x,\bo)\le T\ll1$, as long as the surfaces $\dist(x,\bo)=s$, $s\in [0,T]$ are smooth and still strictly convex. The second one is to show that $F=0$ in a subset of $\Omega$ that can be foliated by strictly convex surfaces starting from ones outside $\Omega$, see Theorem~\ref{thm_uq2}. This only requires Cauchy data on a part of $\bo$, where those surfaces intersect $\bo$,  and proves $F=0$ in a subdomain.  The condition on $T$ is sharp. We give a few examples.

In section~\ref{sec_thermo} we study the non-linear problem of recovery of the speed $c$ in \r{1A} and the linear one of recovery of the source $F$ in \r{w2} described above. The time interval is $(0,T)$ now, but the initial condition $u_t=0$ for $t=0$ in \r{1A}, and the requirement that $a$ in \r{w2} has a sufficiently regular even extension in the $t$ variable allow us to use the results in the previous section. In contrast to the boundary control problems, in the thermoacoustic problem we are given the Dirichlet data on $(0,T)\times\bo$ or on a part of it but no Neumann data. On the other hand,  we know that the solution extends for $x\not\in\Omega$ as a solution again, and the initial data at $t=0$ is zero there. This allows us in Lemma~\ref{lemma_C} to recover the Neumann data from the Dirichlet one in case of data on the whole $\bo$. Then we extend the solution in an even way for $t<0$ and apply the results in section~\ref{sec_un1}. The main requirement is $\Omega$ to have a foliation of strictly convex surfaces, and the time interval $(0,T)$ is sharp. 

The partial data case, with observations on $(0,T)\times\Gamma$, where $\Gamma\subset\bo$ in the thermoacoustic setting is  studied in Theorem~\ref{thm_uq2t}. The main difficulty is the need to recover the Neumann data there as well. Then one applies directly the results in section~\ref{sec_un1}. We show  that one can recover $F$, respectively $c$ in some  neighborhood of $\Gamma$ that might be smaller compared with the case of having  Cauchy data  on the whole $(0,T)\times\Gamma$. There is a new ``cone''  condition that might shrink the domain where we prove $F=0$, or respectively $\tilde c=c$. 

At the end of section~\ref{sec_thermo}, we study the stability of the linear and non-linear problems for \r{1A} and \r{w2}, respectively. As a general principle, for stability, we need to be able to detect all singularities, see \r{cond-stab}. For the linear problem at least, this is a necessary and sufficient condition for stability in any Sobolev spaces, see \cite{SU-JFA}. The corresponding stability estimates are formulated in Theorem~\ref{thm_st} and Theorem~\ref{thm_st2}.

\medskip
\textbf{Acknowledgments.} The authors thank Linh Nguyen and Peter Kuchment for  fruitful discussions on the mathematics of thermoacoustic tomography, and for attracting their attention to reference \cite[Theorem~8.2.2]{Isakov-book}. Thanks are also due to Daniel Tataru for his help to understand better the various unique continuation results. The essential part of this work was done while both authors were visiting the MSRI in Fall 2010.

\section{ A uniqueness result for recovery of a source with one measurement} \label{sec_un1}

Let $\Omega$ be a bounded domain in $\R^n$ with a smooth boundary, and let 
\be{P_def}
P=\partial_t^2- \Delta_g + \sum_j b_j\partial_{x^j} +c 
\ee
be a differential operator in $Q := (-T,T)\times \Omega\subset\R^n$, where $g$ is a smooth Riemannian metric on $\Omega$, and $a_j$, $b$, $c$ are smooth functions on $\bar Q$. 

The level surface $\Sigma=\{\psi=0\}$ of some smooth function $\psi$ is called strongly pseudoconvex w.r.t.\ the hyperbolic operator $P$, if 
\begin{itemize}
  \item[(i)] $\Sigma$ is  non-characteristic, i.e., $|\psi_t|\not= |\d_x\psi|$ when $\psi=0$, and
  \item[(ii)]  $H_p^2\psi>0$ on $T^*\bar\Omega\setminus 0$ whenever $\psi = p= H_p\psi=0$, 
\end{itemize}
where $H_p$ is the Hamiltonian vector field of the principal symbol $p=-\lambda^2+|\xi|^2$ of $P$, and $\lambda$ is the dual variable to $t$, see, e.g., \cite{Tataru04}. 
Here and below, $|\cdot|$ is the norm in the metric $g$ of a covector or a vector. The second condition says that $\Sigma$ is strictly convex w.r.t.\ to the null bicharacteristics of $p$, when viewed from $\psi>0$. In other words, the tangent null bicharacteristics to $\Sigma$ are curved towards $\psi>0$. 

The function $\phi$ with a non-degenerate differential and non-characteristic level set $\phi=0$ is pseudoconvex, if a condition stronger than (ii) is satisfied. We are not going to formulate that condition; it would be enough for our purposes to use the well known fact that if the $\Sigma=\{\psi=0\}$ as above is pseudoconvex, then for $\mu\gg1$,  $\phi =\exp(\mu \psi)-1$ is a pseudoconvex function, non-degenerate on $\Sigma$, and $\Sigma=\{\phi=0\}$; moreover $\{\phi>0\} = \{\psi>0\}$. For details we refer to \cite{Tataru04}.

Let $\phi$ be strongly pseudoconvex in $\bar Q$ w.r.t.\ $P$. Then it is well known that for all $u\in C_0^2(Q)$, 
\be{1.CE}
\tau\int e^{2\tau\phi}\left( u_t^2 +|\nabla u|^2+\tau^2|u|^2 \right)\, \d t\,\d x\le C\int   e^{2\tau\phi}   |Pu|^2\, \d t\,\d x,\quad  \tau>1,
\ee
see \cite{Tataru04, Isakov-book, Isakov04}.  

To reformulate condition (ii) in the tangent bundle, recall that the metric $g$ provides a natural isomorphism between covectors and vectors by the formula $\Phi : (x,\xi) \mapsto (x,v)$, where $ \xi_i = g_{ij}(x)v^j$, in local coordinates, where $v$ is a vector at $x$. For any function $\psi$ on $T^*\Omega$, one gets a function $\Phi^*\psi$ on $T\Omega$. Let $q=|\xi|^2/2$ be the ``$x$ part'' of $p$, rescaled for convenience. 
It is known that $H_q = \Phi_* G \Phi^*$, where $G$ is the generator of the geodesic flow. Also, the energy level $q=1/2$ is pushed forward to the unit sphere bundle $S\Omega$. 

We have $H_{p/2} = -\lambda \partial_t+H_q
$,   therefore, 
\[
\frac14 H_p^2\psi= (\lambda\partial_t-H_q)^2\psi .
\]
Identify covectors with vectors by the map $\Phi$, to get that condition (ii) is equivalent to
\be{1.1}
\psi=0, \quad \psi_t-G\psi=0 \quad \Longrightarrow\quad       (\partial_t - G)^2\psi>0\quad \text{for $(t,x)\in Q$,\, $|\xi|=1$}, 
\ee
and we used the fact that $p=0$ implies $\lambda^2=|\xi|^2$, as well as the homogeneity properties of $G$ w.r.t.\ $\xi$. 

Let us look for $\psi$ of the type
\[
\psi =  r^2(x)-\delta t^2-s , \quad 0<\delta<1,
\]
with $s$ a parameter, 
Then to satisfy (ii), it is enough to have
\be{1.2a}
G^2(r^2/2)>\delta|\xi|^2.
\ee
Since we want eventually to take the limit  $\delta\to1$ to get sharp results, we arrive at the condition
\be{1.2}
G^2(r^2/2)\ge |\xi|^2.
\ee
It is enough to have this inequality in the $x$-projection of $\Sigma$ in $\bar\Omega$, as the first condition in \r{1.1} indicates. Note that this guarantees (ii) only, we still have to choose $r$ so that (i) holds. The latter is equivalent to 
\be{1.2A}
|\d (r^2/2)|\not= \delta|t| \quad \text{on $\Sigma\cap \bar Q$}.
\ee

\begin{example}\label{ex1} 
Let $r(x)=|x-x_0|$ in $\R^n$ with some fixed $x_0$; then $\psi=|x-x_0|^2-\delta t^2-s$. Then 
\[
G^2(r^2/2) = (\xi\cdot\nabla_x)^2|x-x_0|^2/2= |\xi|^2,
\]
and \r{1.2} holds. Condition \r{1.2A} is satisfied for $s>0$  because $|\d (r^2/2)|^2 = |x-x_0|^2=\delta t^2+s>\delta^2 t^2$. 
Then $\Sigma$ is a hyperboloid of one sheet. 
\end{example}

More generally, let $r(x) =\rho(x,x_0)$, where $\rho$ is the distance in the metric $g$. This is the function that has been used  in the Riemannian case. It satisfies \r{1.2a} for $r\ll1$ only, in general. For metrics of negative curvature, there is no restriction, see \cite{Romanov06}.

In this paper, $\partial/\partial \nu$ denotes exterior normal derivative to $\bo$.

\begin{theorem}\label{thm_uq1} Let $Q = (-T,T)\times \Omega$, and 
\be{reg}
\text{$\partial_t^j a \in C(\bar Q)$, $j\le2$,  $F\in L^2(\Omega)$,  and $\partial_t^j \partial^\alpha u\in L^2(Q)$, $j\le3$, $|\alpha|\le1$}. 
\ee
Let $u$ be a (non-unique) solution of  
\begin{equation}   \label{p1}
\left\{
\begin{array}{rcll}
Pu &=& a (t,x)F(x) &  \mbox{in $(-T,T)\times \Omega$},\\
u|_{t=0}&=&0& \mbox{in  $\Omega$}
\end{array}
\right.               
\end{equation}
Let $\phi$ be a strongly pseudoconvex function in $\bar Q$.  
Let 
\be{p1a}
\mathcal{G}\subset (-T,T)\times \bo, \quad \text{$\phi<0$ on $\left((-T,T)\times \bo\right)  \setminus\mathcal{G}$},\quad  \text{and $\phi(t,\cdot)\le \phi(0,\cdot)$\;  for $|t|<T$}. 
\ee 
Let $\supp F\subset K$, where $K\subset\bar\Omega$ is compact, and 
\be{alpha}
 a(0,\cdot)\not=0\quad \text{on $K$}.
\ee
If
\be{p1b}
u=\partial u/\partial\nu=0\quad \text{on $\mathcal{G}$},
\ee
then 
\[
F=0\quad \text{in $\{x;\; \phi(0,x)>0\}$}. 
\]
\end{theorem}

\begin{proof} We follow the proof of  \cite[Theorem~8.2.2]{Isakov-book}. 
Set $Q_\eps = Q\cap \{\phi>\eps\}$. 
Fix $\eps>0$, and let $\chi\in C^\infty$ be such that $\chi=1$ in $Q_\eps$, and $\supp \chi\subset \bar Q_0$. We will apply the Carleman estimate \r{1.CE} to $\partial^j_t\chi u$, $j=0$ and $j=2$, by shrinking $Q_0$ to $Q_\eps$ on the left.   We are using the fact here that $u$ has zero Cauchy data on $\mathcal{G}$. Then $\partial^j_t\chi u$ can be approximated by $C_0^\infty(Q)$ functions in the $H^2$ norms, see also the remark following the proof. 
We have 
\[
P\partial_t^j \chi  u= \partial_t^j \left( \chi P u + [P,\chi]u\right) = \partial_t^j \left(\chi a F+ [P,\chi] u\right) ,
\]
where the commutator $[P,\chi]$ is a   differential operator of order $1$. 
Since $\chi=1$ on $Q_\eps$, we get 
\begin{align}\label{p2}
\tau \int_{Q_\eps}e^{2\tau \phi} \left(\tau^2|u|^2+ \tau^2 |u_{tt}|^2+  |u_{ttt}|^2 \right) \,\d\sigma 
& \le  C\left( \int_Q e^{2\tau \phi}|F|^2 \,\d\sigma+ \int_{Q\setminus Q_\eps} e^{2\tau \phi} \sum_{j=0}^2 \sum_{|\alpha|\le 1}|\partial_t^j \partial_{t,x}^\alpha u|^2\,\d 
\sigma 
\right)\\ \nonumber
&\le  C\int_Q   e^{2\tau \phi}|F|^2   \,\d\sigma+ Ce^{2\tau \eps},
\end{align}
where $\d\sigma= \d t\,\d \Vol(x)$ is the natural measure on $Q$. 
We will estimate the first term in the r.h.s.\ above. 
From  equation \r{p1} and is initial condition, $u_{tt}(0,\cdot) =  a(0,\cdot)f$. By the assumptions on $ a$, $|F|\le C |u_{tt}(0,\cdot)|$. By \r{p1a}, 
\be{p3}
\int_Q e^{2\tau\phi } |F|^2  \, \d \sigma
\le 2T \int_\Omega    e^{2\tau\phi(0,\cdot)} |u_{tt}(0,\cdot)|^2\, \d \Vol(x).
\ee
This integral admits the estimate
\begin{align*}
 \int_\Omega    e^{2\tau\phi(0,\cdot)} |u_{tt}(0,\cdot)|^2\, \d \Vol(x) &= - \int_0^T \int_\Omega \frac{\partial}{\partial s} \left(   e^{2\tau\phi(s,\cdot)}|u_{tt}(s,\cdot)|^2             \right)\d\Vol(x)\, \d s + \int_\Omega e^{2\tau\phi(T,\cdot)} |u_{tt}(T,\cdot)|^2\, \d \Vol(x) \\
 & \le C   \int_Q e^{2\tau\phi} \left(  \tau |u_{tt}|^2 +\tau^{-1}|u_{ttt}|^2       \right)     \d\sigma +C,
\end{align*}
because $\phi(T,\cdot)\le0$, and by the Cauchy inequality. This inequality, together with \r{p3}, estimate the integral of the first term in the r.h.s.\ of \r{p2}. Therefore,
\[
\tau \int_{Q_\eps}e^{2\tau \phi}  \left(\tau^2|u|^2+  \tau^2 |u_{tt}|^2+  |u_{ttt}|^2 \right) \,\d\sigma 
 \le  C\left( 
 \tau^{-1}\int_Q   \left( \tau^2 |u_{tt}|^2+  |u_{ttt}|^2 \right)  \,\d\sigma + e^{2\tau \eps}
\right).
\]
Split the integration on the right into $Q_\eps$ and $Q\setminus Q_\eps$. For $\tau\gg1$, the integral over $Q_\eps$ will be absorbed by the l.h.s. On $Q\setminus Q_\eps$, we have $
e^{2\tau\phi}\le e^{2\tau\eps}$. Therefore,
\[
\tau \int_{Q_\eps}e^{2\tau \phi}  |u|^2\,\d\sigma 
 \le  C e^{2\tau\eps}, \quad \text{for $\tau\gg1$}. 
\]
Thus,
\[
\int_{Q_\eps}e^{2\tau (\phi-\eps)}  |u|^2\,\d\sigma 
 \le  C/\tau , \quad \text{for $\tau\gg1$}. 
\]
Since $\phi-\eps\ge0$ in $Q_\eps$, letting $\tau\to\infty$, we get $u=0$ in $Q_\eps$. Since $\eps>0$ is arbitrary, we get $u=0$ in $Q$. This proves the theorem.
\end{proof}

The following lemma will allow us below to apply the proof of  the theorem to a larger class of non-smooth boundaries. 

\begin{lemma}\label{cor-1}
Let $D\subset\R^n$ be open, and assume that near each point $x_0\in \bo$, $D$ is represented by $y\ge0$ and $z\ge0$, where $y$, and $z$ are functions with non-zero differentials, satisfying the following condition
\[
\text{if $\{y=0\}$ and $\{z=0\}$ intersect, then $\{y=0\}\setminus \{z=0\}$ is dense in $\{y=0\}$.}
\]
Let $u\in C^2(\bar D)$ has Cauchy data $0$ on $\partial D$ in the sense that extended as $0$ outside $D$, it still belongs to $C^2$. Then $u$ can be approximated by $C_0^2(D)$ functions in $H^2(\R^n)$. 
\end{lemma}

\begin{proof}
Let $\chi\in C_0^\infty(\R)$ be such that $\chi=1$ near $0$. Set locally, near $x_0\in \partial D$, 
\[
u_\eps = (1-\chi(y/\eps))(1-\chi(z/\eps))u.
\]
Using a partition of unity, we define such $u_\eps\in C_0^\infty(D)$. We claim that $u_\eps\to u$ in $H^2(\R^n)$, as $\eps\to 0+$.  Take the second derivatives of $u-u_\eps$ to see that we need to show that the following terms converge to $0$ in $L^2(\R^n)$:
\[
\eps^{-1}\chi'(y/\eps)u, \quad \eps^{-2}\chi''(y/\eps)u, \quad \eps^{-2}\chi'(y/\eps)\chi'(z/\eps) u,
\]
where we used the fact that the derivatives of $y$ and $z$ are bounded. We list terms involving lower powers of $\eps^{-1}$ and derivatives of $u$, which analysis is similar. Similarly,  we will not analyze the zero and the first order derivatives of $u_\eps$. 
Since $y/\eps$ and $z/\eps$ are bounded on the support of $\chi'(y/\eps)$ and $\chi'(z/\eps)$, respectively, we may replace the leading coefficient $\eps^{-1}$ in the first term by $y^{-1}$, etc. Since $u=0$ for $z=0$, and $\d y\not=0$, we have $u=zu_1$, where $u_1$ is a smooth function. Next, $u_1=0$ for $y=0$ at least when $z\not=0$. That set of $y$'s however is dense in $\{y=0\}$ by assumption. By continuity, $u_1=0$ for $y=0$. Thus $u_1=yu_2$ with $u_2$ smooth, therefore, $u=zyu_2$. Now, the proof of the claim follows from the fact that $\chi''(y/\eps)$  tends to $0$ in $L^2(\R^n)$; and this is also true if we replace $\chi''$ by  $\chi'$ or $\chi$. 
\end{proof}

We recall next  a unique continuation result due to Tataru \cite{Tataru99}, see also \cite[Theorem~4]{SU-thermo}. Assume that a locally $H^1$ function $u$ solves the homogeneous wave equation $Pu=0$ (near the set indicated in \r{ucT} below) and vanishes in a neighborhood of $[-T,T]\times \{x_0\}$ for some $x_0$ and $T>0$. Then 
\be{ucT}
u(t,x)=0\quad \text{for}\quad |t|+\dist(x_0,x)< T.
\ee
One can formulate a uniqueness continuation statement along a curve in the $x$ space.

\begin{proposition}\label{pr_uc}
Let $[0,T]\ni \sigma \mapsto c(s)$ be a  smooth curve in $\R^n$ so that $c(0)=x_0$ for some $x_0$, and $\sigma$ is a unit speed parameter in the metric $g$. Let $u=0$ near $[-T,T]\times \{x_0\}$. Then 
\be{ucT2}
u(t,x) =0 \quad \text{near}\quad \{(t,x); \; x=c(\sigma), |t|\le T-\sigma  \},
\ee
provided that $u$ solves the homogeneous wave equation $Pu=0$ near that set. 
\end{proposition}

The proof of the proposition is contained in the proof or \cite[Theorem~6.1]{SU-thermo_brain}. The idea is to cover the curve $c$ with small balls and to prove  unique continuation step by step.

Next theorem in fact follows from Theorem~\ref{thm_uq2} below but its proof is simpler, and serves as a basis for the proof of Theorem~\ref{thm_uq2}. We refer to Figure~\ref{fig:carleman5} for an illustration.

\begin{figure}[tbp] 
  \centering
  \includegraphics[bb=7 1 334 215,width=2.39in,height=1.56in,keepaspectratio]{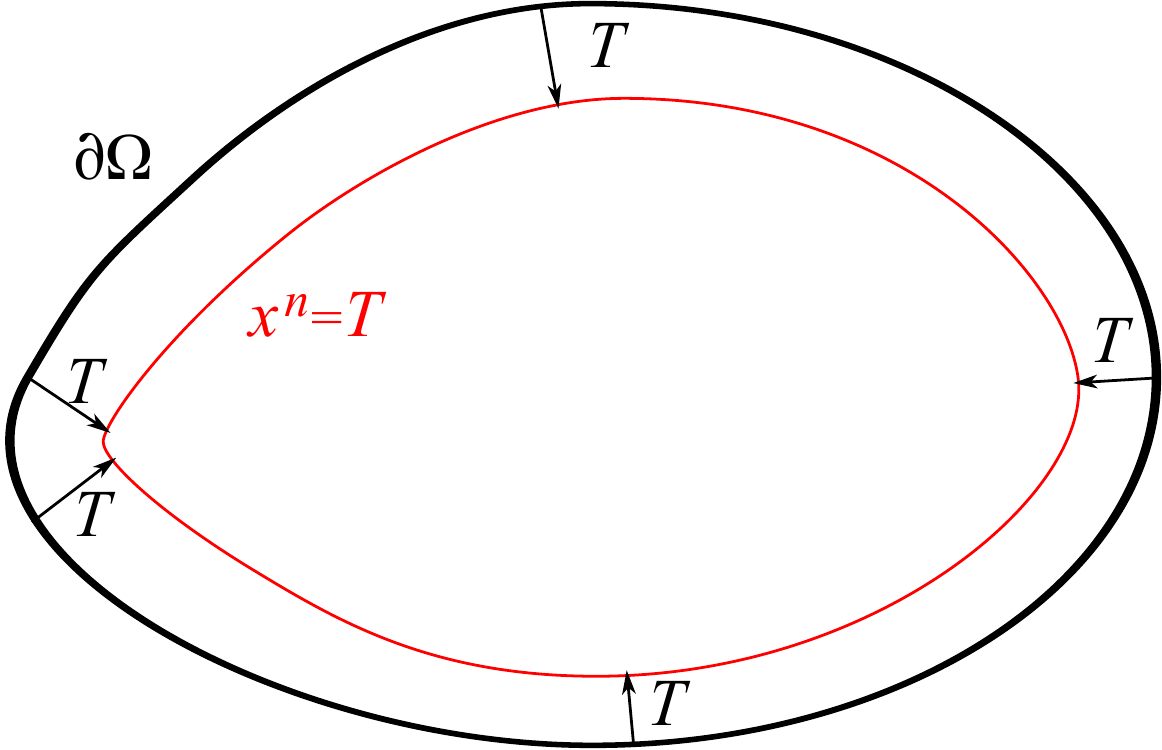}
  \caption{Theorem~\ref{thm_2}}
  \label{fig:carleman5}
\end{figure}

\begin{theorem} \label{thm_2}
Assume that $\bo$ is strictly convex.  
Let $T>0$ be such that $x^n :=\dist(x,\bo)$ is a  smooth function in $\Omega$ with non-zero differential for $0\le x^n\le T$; and $\{x^n=s\}$, $0\le s<T$ are strictly convex surfaces. 
Let $u$ solve \r{p1} and the function $a$ satisfies \r{alpha} for $0\le x^n\le T$. Assume also the regularity conditions \r{reg}. 
Then if   $u$ has zero Cauchy data on $(-T,T)\times \bo$, we also have 
\be{thm2.1}
F(x)=0\quad\text{for $x\in \Omega$, $\dist(x,\bo)<T$}.
\ee
\end{theorem}

\begin{proof}
Let $(x',x^n)$ be normal boundary coordinates near $\bo$ with $s_0$ fixed. Here $x^n$ is the signed distance to $\bo$ so that $x^n>0$ in $\Omega$. The function $x^n$ is defined in a small neighborhood of $\bo$ while $x'$ are local coordinates near some  boundary point.  
The metric $g$ then takes the form
\[
g_{\alpha\beta}(x',x^n)\d x^\alpha\d x^\beta + (\d x^n)^2,
\]
$\alpha, \beta\le n-1$. 
Given $f(x)$, independent of the dual variable $\xi$, we have
\be{p4a}
G^2 f =  \left( \xi^i \frac{\partial}{\partial x^i} - \Gamma_{jk}^i \xi^j\xi^k \frac{\partial}{\partial \xi^i}  \right) \xi^\ell  \frac{\partial}{\partial x^\ell} f = 
\frac{\partial^2 f}{\partial x^i\partial x^j} \xi^i \xi^j  - \Gamma_{jk}^i \xi^j\xi^k  \frac{\partial f}{\partial x^i} .
\ee
Here, $\Gamma_{jk}^n = -\frac12\partial g_{jk}/\partial x^n$ is the second fundamental form $\II>0$ of the level sets of $x^n$,  w.r.t.\ the chosen orientation,  and it is zero when either $i=n$ or $j=n$. 
Assume now that $f$ is a function of $x^n$ only. 
Restrict \r{p4a} to $\bo=\{x^n=0\}$ to get
\[
G^2 f\big|_{x^n=0} = \left(
\frac{\partial^2 f}{\partial (x^n)^2}   (\xi^n)^2 -\II(x')(\xi',\xi') \frac{\partial f}{\partial x^n} \right)\Big|_{x^n=0}.
\]
To satisfy \r{1.2} for $f=r^2/2$, we need
\be{p5}
\frac{\partial^2 f}{\partial  (x^n)^2}   \ge1, \quad  -\frac{\partial  f}{\partial x^n}\ge R\quad \text{for $x^n=0$},
\ee
where $1/R$ is the minimum over $\bo$ of the smallest eigenvalue (principal curvature) of the second fundamental form $\II$. We can think of $R$ as the largest curvature radius of $\bo$; and by assumption, $0<R<\infty$.  Then the following function satisfies \r{1.2}:
\[
r(x) = R-x^n,
\]
because then  $f := r^2/2= (R-x^n)^2/2$ clearly  satisfies \r{p5}. 
Therefore, the function 
\be{p_phi}
\psi_s(t,x) := (R-x^n)^2-\delta t^2-s
\ee
is strongly pseudoconvex, assuming also that $\psi_s=0$ is non-characteristic. Also, the last inequality in \r{p1a} holds. 
 Note that in Example~\ref{ex1}, if $\bo=\{x;\; |x|=R_0\}$, then $x^n= R_0-|x|$, and one can choose $R=R_0$. Then $\psi=|x|^2-\delta t^2-s$, which is the phase function in the example. As in the example, we can show that $\{\psi_s=0\}$ is non-characteristic for $s>0$. In fact, we will work locally near $x^n=t=0$, and $s=R$, and clearly, $\psi_s$ is non-characteristic there.

Fix $0<\eps\ll1$, and let  $0<x^n<\eps$.  We restrict $s$ to the interval $(R-\eps)^2\le s\le  R^2$.  This choice of the phase function corresponds to pseudoconvex surfaces given by 
\be{p6}
(R-x^n)^2-\delta t^2 = s, \quad (R-\eps)^2\le s\le  R^2, \quad 0\le x^n \le \eps<R. 
\ee
In $\R_t\times\bar\Omega_x$, this restricts $t$ to $|t|=O(\sqrt{\eps})$, see also \r{p7} below.  
We will apply Theorem~\ref{thm_uq1} with a the phase function $\phi_s :=\exp(\mu\psi_s)-1$, $\mu\gg1$,  with $s$ as in \r{p6}. On $\mathcal{G} := (-T,T)\times\bo$, we have $\psi_s|_{x^n=0}= R^2-\delta t^2-s$. To have $\phi_s<0$ outside $\mathcal{G}$, we need $\psi_s<0$ there, and therefore $R^2-\delta T^2<s$ for 
all  $s$ as in \r{p6}. Therefore,  if
\be{p7}
T>\sqrt{ R^2-(R-\eps)^2 } =\sqrt{2\eps R} +O(\eps^2),
\ee
always true when $\eps\ll1$, 
we can  apply the theorem for $\delta<1$ close enough to $1$. Therefore, we get that if $u$ has zero Cauchy data on $(-T,T)\times\bo$ with $T$ as in \r{p7}, then $F(x)=0$ for $\psi_s(0,x)>0$ for any $s$ as before, i.e., for $0\le x^n<\eps $. Note that this does not prove the theorem yet, even when $T$ is small enough so that we can have equality in \r{p7} with some $\eps$ satisfying the smallness requirements. The reason is that we get $F=0$ in a much smaller region: $0\le x^n\le \eps$ instead of $0\le x^n\le T$ because 
for small $\eps$, we have $T\sim \sqrt{\eps}\gg\eps$.  Also, $T\gg1$ when $R$ is large, i.e., when $\bo$ is close to a flat surface at some point and direction.  In other words, the price that we pay with $T$ to push $\supp F$ by $\eps$ depends on the (largest) radius of the curvature of $\bo$, and this is not what we are trying to prove. We will use this argument as an incremental step only, and will prove the theorem by applying a finite number of such steps.

To get the sharp time $T$, not necessarily small,  we notice that we just proved that if  $u$ has zero Cauchy data on $(-T,T)\times\bo$ with $T$ as in \r{p7} and  $\eps\ll1$, then $F(x)=0$ for $x\in \Omega$, $\dist(x,\bo)<\eps$. Then $Pu=0$ in the same domain and $|t|<T$, by \r{p1}.  By  unique continuation, see Proposition~\ref{pr_uc}, 
\be{p7b}
 u(t,x)=0 \quad\text{for $x\in\Omega$, $\dist(x,\bo)+|t|<T$, $\dist(x,\bo)\le \eps$}.
\ee
In particular,
\be{p8}
\partial_x^\alpha u(t,x)=0 \quad\text{for $x\in\Omega$, $\dist(x,\bo)=\eps$, $|t|<T-\eps$, $|\alpha|\le1$},
\ee
provided that $T>\eps$. 

Let $\tilde\eps$ be the supremum of all $\eps$ for each the following statement holds: if $u$ has zero Cauchy data on $(-T,T)\times\bo$, then $F(x)=0$ when  $\dist(x,\bo)<\eps$. Then $\tilde \eps$ has that property as well. 
If $\tilde\eps<T$, then by the argument above, see \r{p8}, $u$ has zero Cauchy data on $(-T+\tilde\eps,T-\tilde\eps) \times\bo$. Then we can repeat the argument leading to \r{p7b} to reduce $\supp f$ even further, and this would contradict the choice of $\tilde\eps$. Therefore, $\tilde \eps=T$. 
\end{proof}

Recall that given two subsets $A$ and $B$ of a metric space, the distance $\dist(A,B)$ is defined by 
\be{dist}
\dist(A,B) =  \sup(\dist(a,B);\; a\in A).
\ee
This function is not symmetric in general, and the Hausdorff distance is defined as 
\[
\dist_{\rm H}(A,B) = \max\left( \dist(A,B), \dist(B,A)           \right).
\]
Let $\Omega_1\Supset\Omega$ be another domain so that $\bo_1$ is given by $\dist(x,\bo)=\eps\ll1$. 
Let $\Sigma_s$, $s_1\le s\le s_2$ be a continuous family of compact oriented surfaces in $\Omega_1$. 
Examples include also surfaces that are not closed in $\bar\Omega$ but can be extended as closed ones in the larger domain $\Omega_1$. By a continuous family, we mean a family so that the Hausdorff distance $\dist_{\rm H}(\Sigma_s,\Sigma_{s_0})$ tends to $0$, as $s\to s_0$, $\forall s_0$. We assume that each $\Sigma_s$ divides $\Omega_1$ in two (open) parts:  
one, in the direction of the normal giving the orientation, that  we denote by $\Sigma_s^\text{int}$; and another one that contains $\bo_1$, that we denote by $\Sigma_s^\text{ext}$.

Let $\Gamma\subset\bo$ be a relatively open subset of $\bo$. Set 
\be{i1}
\mathcal{G} := \left\{  (t,x); \; x\in \Gamma, \, 0<t<\tau(x)    \right\},
\ee
where $\tau$ is a fixed continuous function on $\Gamma$. This corresponds to measurements taken at each $x\in\Gamma$ for the time interval $0<t<\tau(x)$. One special case  is  $\tau(x)\equiv T$, for some $T>0$; then $\mathcal{G}= [0,T]\times\Gamma$. 

\begin{figure}[h] 
  \centering
  \includegraphics[bb=0 0 709 214,width=4.76in,height=1.44in,keepaspectratio]{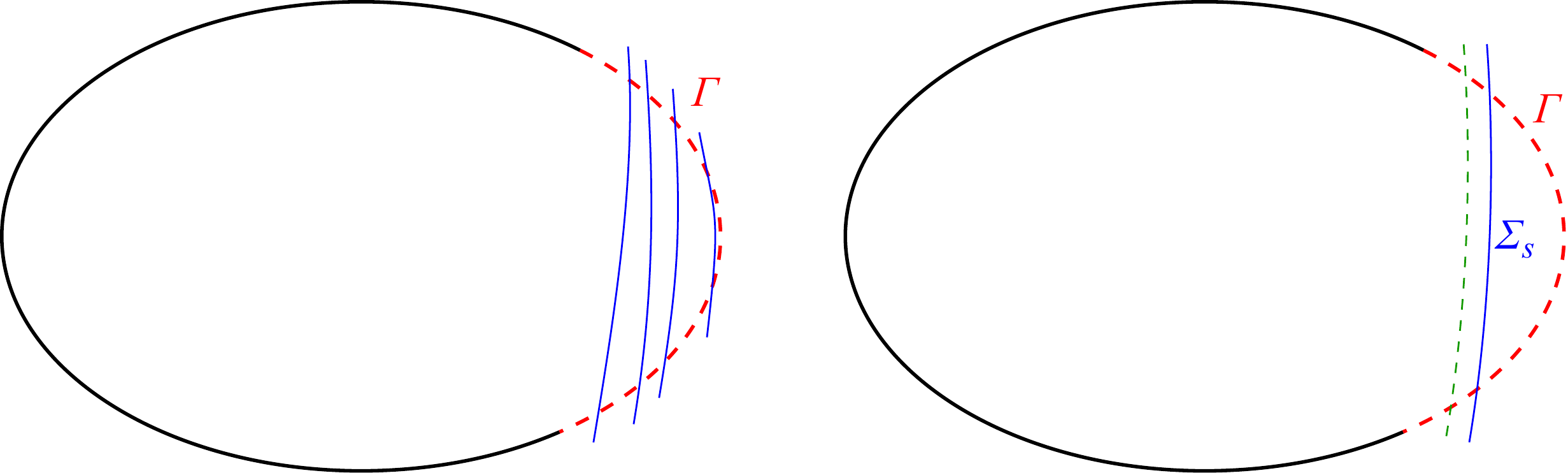}
  \caption{An illustration to Theorem~\ref{thm_uq2} and its proof. Left: the family $\Sigma_s$. Right: The ``inductive'' step of the proof. The r.h.s.\ curve is $\Sigma_s$; the doted one to the left is at distance $\eps$ from  $\Sigma_s$.}
  \label{fig:carleman4}
\end{figure}

\begin{theorem}\label{thm_uq2}
Let $\Omega_1$, $\mathcal{G}$, and $\Sigma_s$ be as above. Let $u$ solve \r{p1} and have zero Cauchy data on $\mathcal{G}$. Let the regularity conditions \r{reg} and the  ellipticity condition \r{alpha} be satisfied on $K := (\cup \Sigma_s)\cap\bar\Omega$.  Assume that 
\begin{itemize}
  \item [(a)] 
  $\supp F\subset \Sigma_{s_1}^\text{\rm int}$, 
  \item[(b)] $\Sigma_s\cap \bar\Omega$ is strictly convex for any $s$, 
  \item[(c)] $\Sigma_s \cap (\bo\setminus \Gamma) = \emptyset$ for any $s$.
\end{itemize}
Assume that  
\be{thm_uq_cX}
\text{for any $x\in \Sigma_{s}\cap \bar\Omega$, there is $y\in \Gamma$ so that  $\tau(y)> \dist(x, y)$}. 
\ee
Then 
\be{thm2.2}
\supp F \cap \Sigma_s=\emptyset \quad \forall s.
\ee
\end{theorem}

\begin{proof} In this proof, we regard $F$ as a function supported in $\bar\Omega$.  
Fix  $s\in [s_1,s_2]$, and assume that 
\be{p9}
\supp F\subset \overline{\Sigma_{s}^\text{int}}.
\ee
Then   for any $x\in \Sigma_{s}\cap \bar\Omega$, $\tau(x)>  \dist(x, y)$ for some $y\in\Gamma$.   
By the unique continuation statement of Proposition~\ref{pr_uc}, $u$ has zero Cauchy data on $\Sigma_{s}\cap \bar\Omega$ for $|t|\ll1$, and by assumption, this is also true on $\Gamma$.

Let $x^n$ be a boundary normal coordinate to $\Sigma_{s}$ positive in $\Sigma_{s}^\text{int}$. Let $\psi(t,x)$ be as in \r{p_phi}, and $\eps>0$ be as in the proof of Theorem~\ref{thm_2}. The function $\phi=\exp(\mu\phi)$, $\mu\gg1$, is guaranteed to be pseudoconvex only for $x$ in an $O(\eps)$ neighborhood of $\Sigma_{s}\cap  \bar\Omega$, and for $|t|=O(\sqrt{\eps})$, if $\eps\ll1$, by (b). We  apply Lemma~\ref{cor-1} to the set $D = \Sigma_{s}^\text{int}\cap \Omega$ to 
conclude by (b) and (c), that $F=0$ in some neighborhood of $\Sigma_{s}\cap \bar\Omega$, see Figure~\ref{fig:carleman4}.           

Let now $s_0$ be the supremum of all $s\in [s_1,s_2]$ for which \r{p9}  holds. The latter set is non-empty, by (a). By the continuity of $s\mapsto \Sigma_s$, \r{p9} holds for $s=s_0$. Indeed, assuming that \r{p9} does not hold for $s=s_0$, we can find $0<\eps\ll1$,  so that \r{p9}  does not hold for $s_0-\eps\le s\le s_0$. That contradicts the choice of $s_0$ to be the least upper bound. On the other hand, by what we proved above, $s_0$ cannot be an upper bound, unless $s=s_2$. This completes the proof.
\end{proof} 

\begin{remark}\label{rem_cond0}
The meaning of condition \r{thm_uq_cX} is to guarantee that any point $x$ where we want to prove $F(x)=0$ is reachable from $\Gamma$ (from some point $y$) at a time not exceeding $\tau(y)$. In other words, there is a ``signal'' (a unit speed curve) issued from $x$ that will reach the observation part $\Gamma$ of $\bo$ at a time while we are still making measurements there. By finite speed of propagation, it is a necessary condition, as well.
\end{remark} 

\begin{remark}\label{rem_cond}
A sufficient but an easier to formulate condition to replace \r{thm_uq_cX} is  
\be{thm_uq_c2}
\mathcal{G}= [0,T]\times\Gamma\quad \text{with}\quad T> \max_s \dist(\Sigma_{s}\cap \bar\Omega, \Gamma),
\ee
see \r{dist}. An even simpler sufficient condition is
\be{thm_uq_c3}
\mathcal{G}= [0,T]\times\Gamma\quad \text{with}\quad T>   \dist( \Omega, \Gamma).
\ee
\end{remark}

\subsection{Examples} 
\begin{example} \label{ex1a}
Let $\Omega\subset \R^2$ be a bounded domain, and let $g$ be a metric on $\bar\Omega$. Assume that $\bo $ is convex in the metric, and that there is $x_0\in \bo$, so that all geodesics issued from $\bo$, pointing into $\bar\Omega$, exit $\bar\Omega$ after some fixed time. That property does not depend on the way we extend $g$ outside $\bar\Omega$. All simple (see \cite{SU-JAMS} for a definition) $(\Omega,g)$ have this property, and all non-trapping convex ones have it, too. We will show that in this case we can cover  $\bar\Omega$ by a foliation of smooth surfaces (curves, actually) $\Sigma_s$ that are a small perturbation of the geodesics through $x_0$. 

\begin{figure}[h] 
  \centering
  \includegraphics[bb=10 0 376 240,width=2.5in,height=1.64in,keepaspectratio]{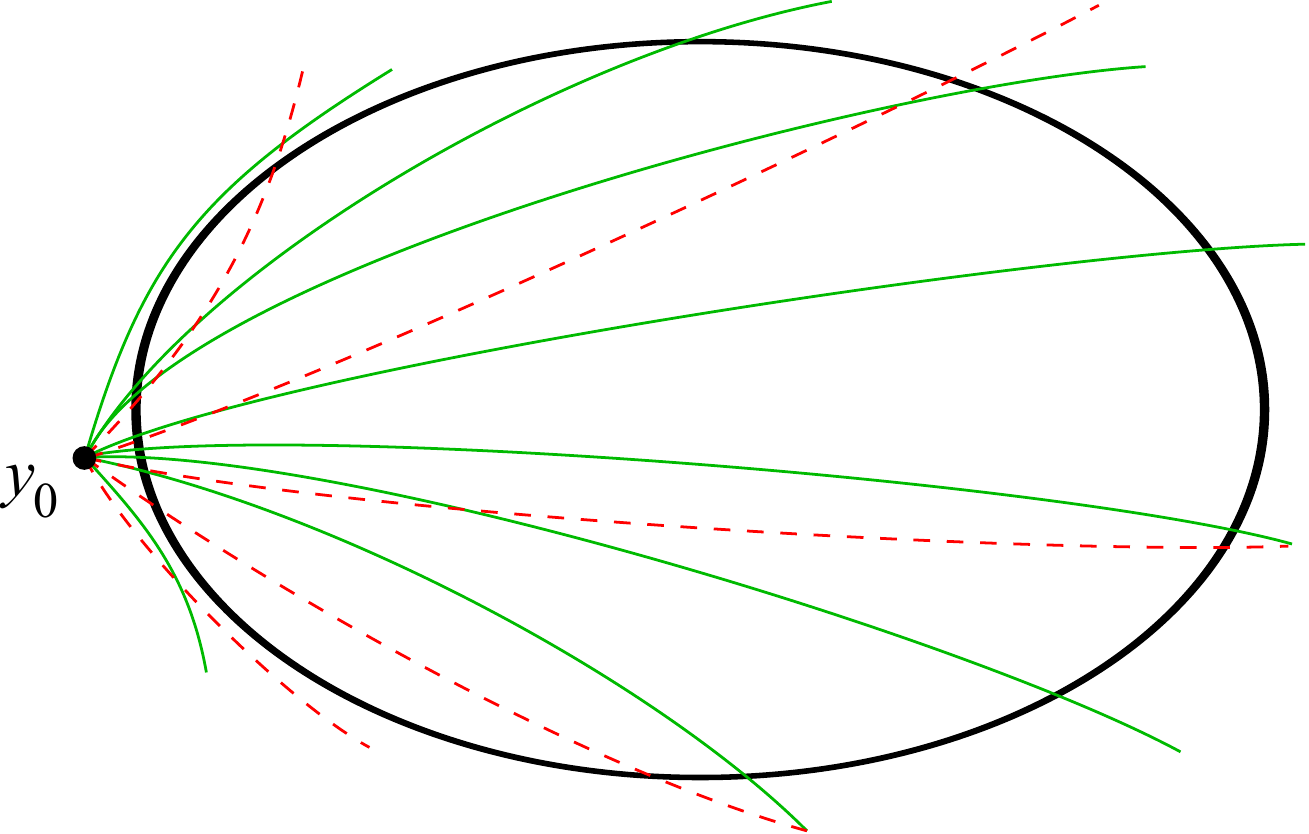}
  \caption{Example~\ref{ex1a}}
  \label{fig:carleman1}
\end{figure}

Extend $g$ in a small neighborhood $\Omega_1$ of $\bar\Omega$, and let $y_0\not\in \bar\Omega$ be close to $x_0$ so that the geodesics through $y_0$ have the same property but in $\Omega_1$. Choose global coordinates in the latter as normal coordinates centered at $y_0$. In those coordinates, the geodesics through $y_0$ are the lines, i.e., they solve $x''=0$. Let at each $x$,  $J$ be the rotation operator in the tangent space given by $Jv= (-v_2, v_1)$, where we used the standard index raising/lowering  convention, and vector $Jv$ is identified with the covector on the r.h.s.  Given $0<\delta\ll1$, define the curves $\Sigma_s$ as solutions of $x''=-\delta Jx'$. The parameter $s$ measures the angle of the initial direction at $y_0$ with a fixed direction. Then $\Sigma_s$ are strictly convex when viewed from the side determined by the normal  $Jx'$. In Figure~\ref{fig:carleman1}, this is the upper side. We can always extend the curves $\Sigma_s$ to closed ones with the extension being outside $\Omega$.  Assume now that $u$ has zero Cauchy data. Then we can apply Theorem~\ref{thm_uq2} to conclude that $F=0$ when $T$ is appropriately chosen. A possible choice of $T$ is the diameter of $\bar\Omega$, given by $\max(\dist(x,y);\; x,y\in \bar\Omega)$.  To optimize $T$, we can consider a similar  family, with $\delta$ negative. In  Figure~\ref{fig:carleman1}, they are shown as dashed curves. 
Then the ``interior'' and the ``exterior'' are reversed. The two families converge to the set of the geodesics through $y_0$, as $|\delta|\to0$. Then the value for $T$ enough to apply the theorem can be obtained by finding  a geodesic $\gamma_0$ through $y_0$ so that maximum of the distances from  $\gamma_0\cap \bar\Omega$ to the upper and the lower side of $\bo$ is maximized; then $T$ is that value. 

Let $\Omega$ be an ellipse, and let $g$ be Euclidean. If $y_0$ is one of the vertices on the major (minor) axis, then $\gamma_0$ is the major (minor) axis, and it is enough to take $T$  to be a half of the length of the major (minor) axis. 
 \end{example}

\begin{example}\label{ex2}
Let $\Omega\subset\R^2$ be as above. Assume that there exists a closed non-self-intersecting geodesic $\gamma_0$ of the metric $g$. Assume that the region between $\bo$ and $\gamma_0$ can be foliated by a continuous family of strictly convex curves $\Sigma_s$. Then $\supp F$ is contained inside $\gamma_0$, if  $T=\dist(\gamma_0,\bo)$. Our analysis does not allow us extend the equality $F=0$ inside.

\begin{figure}[h] 
  \centering
  \includegraphics[bb=12 12 303 206,width=2.3in,height=1.53in,keepaspectratio]{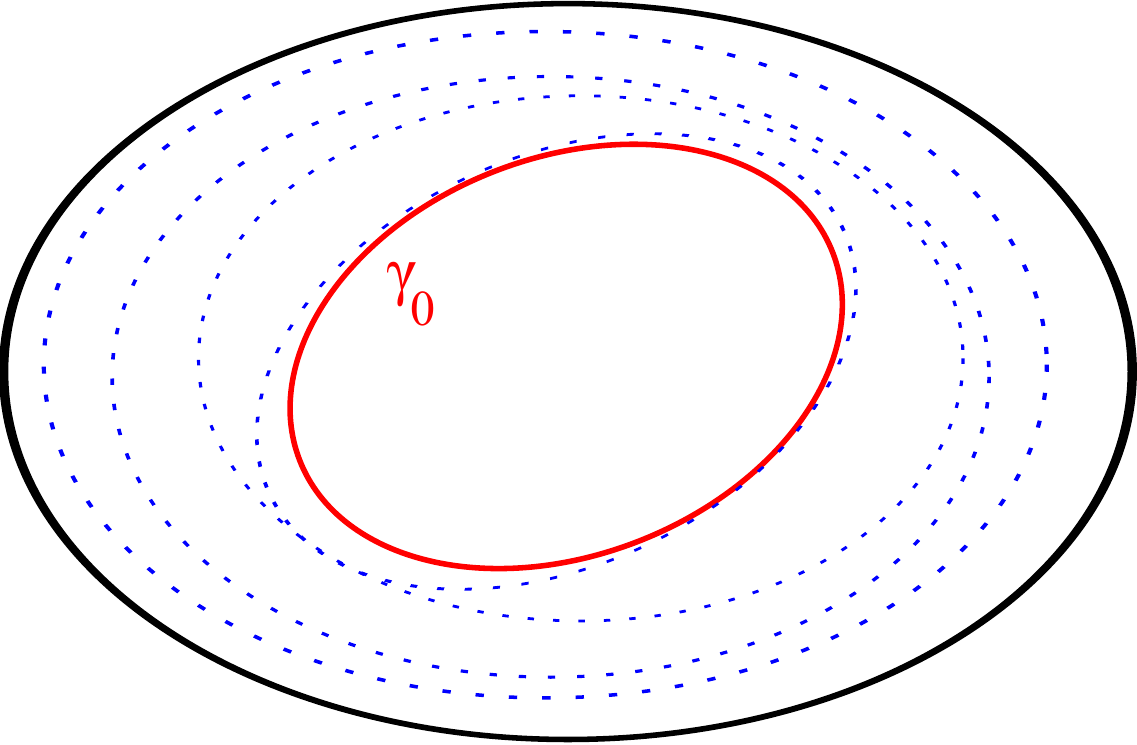}
  \caption{Example~\ref{ex2}}
  \label{fig:carleman2}
\end{figure}
\end{example}

\begin{example}  \label{ex_5} 
Let $\Omega\subset\R^2$ be diffeomorphic to a disk, and let $\Gamma$ be a relatively open connected part of $\bo$. Fix a metric $g$ in some neighborhood of $\bar\Omega$. We do not need to assume that the whole $\bo$ is convex but we will assume that there is a continuous family of geodesics, with endpoints outside $\bar\Omega$, covering the region between $\Gamma$ and the geodesic connecting the endpoints of $\Gamma$, see also Figure~\ref{ex1a}. In Figure~\ref{fig:carleman3}, this is the unshaded region. The latter assumption is fulfilled if for example one of those points has the property that all geodesics issued from it, and pointed inside $\Omega$, leave the unshaded region after some fixed time. In particular, $(\Omega,g)$ being non-trapping suffices. Then we can perturb those geodesics to curves that are convex (bending to the shaded region), as in Example~\ref{ex1a} to show that when $T$ is chosen in an appropriate way, $\supp F$ must be in the shaded region. 

\begin{figure}[h] 
  \centering
  \includegraphics[bb=0 0 327 215,width=2.5in,height=1.64in,keepaspectratio]{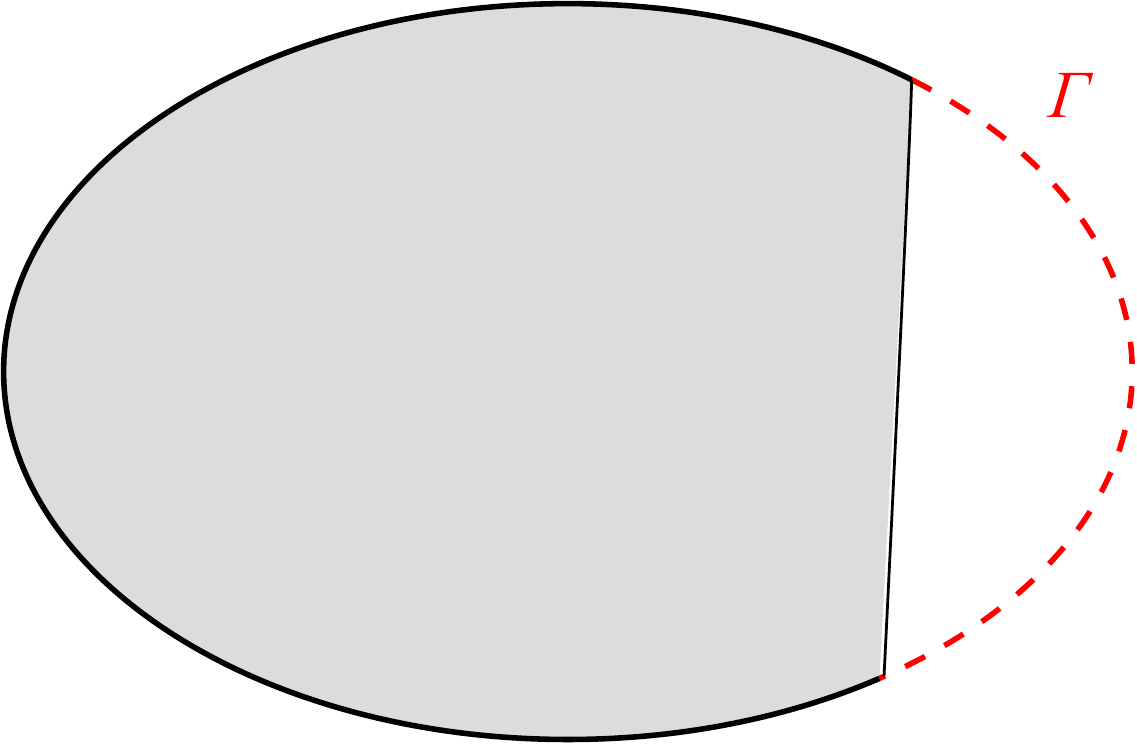}
  \caption{Example~\ref{ex_5}}
  \label{fig:carleman3}
\end{figure}

An higher dimensional analog of this example would be $\Omega$ diffeomorphic to a ball with  $\Gamma\subset\bo$ diffeomorphic to a disk on $\bo$. Then $F=0$ in the region covered by families of convex surfaces. For example, let $\Omega$ be the ellipsoid $\sum (x^i)^2/a_i^2=1 $, and let $g$ be Euclidean. Let $\Gamma= \bo\cap \{x^1>C\}$ with $0<C<a_1$. Then  $F=0$ in $\Omega\cap\{x^1>C\}$, and it is enough to choose $T=a_1-C$. 
That $T$ may or may not be sharp, depending on all $a_j$ and $C$. 
\end{example}

\section{A non-linear problem of recovery of a speed with one  measurement. Applications to thermoacoustics}\label{sec_thermo} 
In section~\ref{sec_un1}, we showed that one can uniquely recover $f$ when $t$ varies over a symmetric interval $[-T,T]$, and $u|_{t=0}$ is known. No knowledge of $u_t|_{t=0}$ is required. 
Assume now that $t$ varies over the interval $[0,T]$, and $a$, $u$ admit even extensions for $t\in [-T,T]$ of  regularity as in the preceding section. In particular, this means  that $u_t|_{t=0}=0$. In other words, $u$ solves
\begin{equation}   \label{p1'}
\left\{
\begin{array}{rcll}
Pu &=&a(t,x)F(x) &  \mbox{in $(0,T)\times \Omega$},\\
u|_{t=0}&=&0& \mbox{in  $\Omega$},\\
u_t|_{t=0}&=&0& \mbox{in  $\Omega$},
\end{array}
\right.               
\end{equation}
compare with \r{w2}. 
Then one has obvious corollaries of the results of the previous section that we are not going to formulate. One can interpret those results as a recovery of a source, given $a$. If $a=1$, then one can differentiate the equation above w.r.t.\ $t$ and to reduce the problem to recovery of an initial condition for the homogeneous wave equation, that is the classical thermoacoustic problem of recovering $f$ given $\Lambda f$, see \r{1b} below, with a known speed.

\subsection{The thermoacoustic model}
Let $u$ solve the problem \r{1A} 
where $c=c(x)>0$ is smooth and $T>0$ is fixed. Note that the wave equation is solved in the whole $\R^n$ now. 

Assume that $f$ is supported in $\bar\Omega$, where $\Omega\subset \R^n$ is some smooth bounded domain. Assume also that $c=1$ outside $\Omega$. This is not an essential assumption; it is enough $c$ to be known and fixed outside $\Omega$. 
The measurements are modeled by the operator
\be{1b}
\Lambda f : = u|_{[0,T]\times\partial\Omega}.
\ee
The problem is to reconstruct the unknown $c$ and $f$, if possible. We will study now the case when $f$ is known, and we want to reconstruct $c$.

\subsection{Uniqueness results for the linear problem}
Let $(c,f)$, $(\tilde c,\tilde f)$ be two pairs, and let $u$, $\tilde u$  be the corresponding solutions of \r{1A}. Then
\begin{equation}   \label{w1}
\left\{
\begin{array}{rcll}
(\partial_t^2 -c^2\Delta)(\tilde u-u) &=&(\tilde c^2-c^2)\Delta \tilde u &  \mbox{in $(0,T)\times \R^n$},\\
(\tilde u-u)|_{t=0} &=&0,\\ \quad \partial_t (\tilde u-u)|_{t=0}& =&0.
\end{array}
\right.               
\end{equation}
Then 
\[
\big( \tilde \Lambda    - \Lambda \big) f = (\tilde u-u)\big|_{[0,T]\times \bo}.
\]
We have 
\be{w1a}
(\delta\Lambda) f :=  \big( \tilde \Lambda    - \Lambda \big) f =   w\big|_{[0,T]\times \bo},
\ee
where $w$ solves \r{w2} with  $F$ and $a$ as in \r{Fa}. 
Then $\supp F\subset\bar\Omega$, and  given the regularity of $c$ and $\tilde c$, we also have $F=0$ on $\bo$. 
The measurement \r{w1a} however determines the Dirichlet data on $[0,T]\times\bo$ only. The Neumann data can be recovered from that however, see also \cite[sec.~7]{SU-thermo_brain}, and Lemma~\ref{lemma_C} below. We emphasize again that $w$ solves the wave equation for $x$ in the whole $\R^n$, and this is what allows us to recover the Neumann data. 

We assume below that $w$ solves the more general problem
\begin{equation}   \label{w2A}
\left\{
\begin{array}{rcll}
(\partial_t^2 -\Delta_g)w &=&a(t,x)F(x)&  \mbox{in $(0,T)\times \R^n$},\\
w|_{t=0} &=& 0,\\ \quad \partial_t w|_{t=0}& =&0,
\end{array}
\right.               
\end{equation}
where $g$ is a smooth Riemannian metric that is Euclidean outside $\Omega$. 
In some of the results below, we do not assume that $F$ and $a$ are given by \r{Fa}.

We introduce the energy space associated  with the wave equation below, to be able to deal both with metrics and variables speeds. Write the metric $g$ as $g=c^{-2}g_0$, where $c(x)>0$. In the thermoacoustic case, one usually takes $g_0$ to be Euclidean. 
Let $\Delta_{g_0}$ be the Laplace-Beltrami operator as above but related to $g_0$.   
Modulo lower order terms, $c^{2}\Delta_{g_0}$ and $\Delta_g$ coincide. 
Write $P$ in the form
\be{P1}
P=\partial_t^2-A, \quad A= c^2\Delta_{g_0}+ \text{lower order terms},
\ee
compare with \r{P_def}. 
Notice first that $c^2\Delta_{g_0}$ is formally self-adjoint w.r.t.\ the measure $c^{-2}\d \Vol$.   
Given a domain $U$, and a function $u(t,x)$, define the energy
\[
E_U(t,u) = \int_U\left( |\nabla_x u|_0^2   +c^{-2}|u_t|^2 \right)\d \Vol_0(x),
\]
where $|\nabla_x u|_0$ is the norm in the metric $g_0$, and $\d\Vol_0$ is the volume measure w.r.t.\ $g_0$ as well. 
In particular, we define the space $H_{D}(U)$ to be the completion of $C_0^\infty(U)$ under the Dirichlet norm
\be{2.0H}
\|f\|_{H_{D}}^2= \int_U |\nabla_x u|_0^2 \,\d \Vol_0(x).
\ee
It is easy to see that $H_{D}(U)\subset H^1(U)$, if $U$ is bounded with smooth boundary, therefore, $H_{D}(U)$ is topologically equivalent to $H_0^1(U)$. If $U=\R^n$, this is true for $n\ge3$ only.  By the finite speed of propagation, the solution with compactly supported Cauchy data always stays in $H^1$ even when $n=2$.  
The energy norm for the Cauchy data $[f_1,f_2]$, that we denote by $\|\cdot\|_{\mathcal{H}}$ is then defined by
\[
\|[f,f_2]\|^2_{\mathcal{H}} = \int_U\left( |\nabla_x f_1|_0^2    +c^{-2}|f_2|^2 \right)\d \Vol_0(x).
\]
This defines the energy space 
\[
\mathcal{H}(U) = H_D(U)\oplus L^2(U).
\] 
Here and below, $L^2(U) = L^2(U; \; c^{-2}\d \Vol_0)$. Note also that 
\be{Pf}
\|f\|^2_{H_D} = (-A f,f)_{L^2}.
\ee
The wave equation then can be written down as the system
\be{s1}
\mathbf{u}_t= \mathbf{A}\mathbf{u}, \quad \mathbf{A} = \begin{pmatrix} 0&I\\A&0 \end{pmatrix}, 
\ee
where $\mathbf{u}=[u,u_t]$ belongs to the energy space $\mathcal{H}$. The operator $\mathbf{A}$ then extends naturally to a skew-selfadjoint operator on $\mathcal{H}$. We denote by $\mathbf{U}(t)$ the group $\exp(t\mathbf{A})$. 
In this paper, we will deal with either $U=\R^n$ or $U=\Omega$. In the latter case, the definition of $H_D(U)$ reflects Dirichlet boundary conditions.


We will define next  the outgoing DN map. Given $g\in C_0^\infty((0,\infty)\times \bo)$, let $w$ solve the exterior mixed problem with $c=1$:
\begin{equation}   \label{1w}
\left\{
\begin{array}{rcll}
(\partial_t^2 -\Delta)v &=&0 &  \mbox{in $(0,T)\times \R^n\setminus\bar\Omega$},\\
v|_{[0,T]\times\bo}&=&g,\\
v|_{t=0} &=& 0,\\ \quad \partial_t v|_{t=0}& =&0.
\end{array}
\right.               
\end{equation}
Then we set
\[
Ng = \frac{\partial w}{\partial\nu} \Big|_{[0,T]\times\bo}.
\]
By  \cite{LasieckaLT}, for $g\in H^1_{(0)}([0,T]\times \bo)$, we have $[w,w_t]\in C([0,T);\;  \mathcal{H}(\Omega))$; therefore, 
\[
N: H^1_{(0)}([0,T]\times \bo) \to C([0,T]\times H^\frac12(\bo))
\]
is continuous, where the subscript $(0)$ indicates the closed subspace of functions vanishing at $t=0$. Note  that the results in  \cite{LasieckaLT} require the domain to be bounded but by finite domain of dependence we can remove that restriction in our case.  We also refer to \cite[Proposition~2]{finchPR} for  a sharp domain of dependence result for exterior problems.  

When $F$ and $a$ are given by \r{Fa}, the next lemma follows directly from its version  \cite[sec.~7]{SU-thermo_brain} for $\Lambda f$ by subtracting $\tilde\Lambda f$ and $\Lambda f$.  

\begin{lemma}\label{lemma_C}
Let $w$ solve \r{w2} with $F$  supported in $\bar\Omega$,  and let $t\mapsto a(t,\cdot)F(\cdot)\in L^2(\Omega)$ be continuous.  Assume that $c=1$ 
outside $\Omega$. Then for any $T>0$,  $w|_{[0,T]\times \bo}$ determines uniquely  the normal derivative of $w$ on $[0,T] \times \bo$ as follows:
\be{N1}
\frac{\partial w}{\partial\nu}\Big|_{[0,T]\times\bo} = N\left( w|_{[0,T]\times \bo}\right).
\ee
\end{lemma}

\begin{proof}
Let $v$ be the solution of \r{1w} with $g=w|_{[0,T]\times \bo}$. The latter is in $H^1_{(0)}([0,T]\times \bo)$. Let $w$ be the solution of \r{w2}. Then $v-w$ solves the unit speed wave equation in $[0,T]\times \R^n\setminus\Omega$ with zero Dirichlet data  and zero initial data. Therefore, $v=w$ in $[0,T]\times \R^n\setminus\Omega$. 
\end{proof}

With this in mind, we have the following version of Theorem~\ref{thm_2} in this context. Notice that in the two theorems below, we do not assume $F$ and $a$ to be given by \r{Fa}.

\begin{theorem} \label{thm_2t} 
Assume that $\bo$ is strictly convex,  $w$ solves \r{w2} and the function $a$ satisfies the elliptic condition \r{alpha} in the closure of the set \r{thm2.1t} below. Let $a$ and $w$ admit  even extensions satisfying \r{reg}. 
Let $T>0$ be such that $x^n :=\dist(x,\bo)$ is a well defined smooth function in $\Omega$ with non-zero differential; and $\{x^n=s\}$, $0\le s<T$ are strictly convex surfaces. 
Then if   $w=0$ on $[0,T]\times \bo$, we also have 
\be{thm2.1t}
F(x)=0\quad\text{for $x\in \Omega$, $\dist(x,\bo)<T$}.
\ee
\end{theorem}
\begin{proof}
We first apply Lemma~\ref{lemma_C} to conclude that the normal derivative of $w$ vanishes on $[0,T]\times \bo$, as well. 
The assumptions of the theorem  imply that $a(t,x)$ can be extended in an even way satisfying the assumptions of Theorem~\ref{thm_2}. Then  $w$ admits an even extension, as a solution of the wave equation, and thus we  apply Theorem~\ref{thm_2}. 
\end{proof}

Since recovery of $F$ in \r{w2} from $w|_{[0,T]\times\bo}$ is a linear problem, we also get uniqueness for that problem in the set \r{thm2.1t}. We also get unique recovery of the speed $c$ in the region in \r{thm2.1t}. Those are in fact partial cases of Theorem~\ref{thm_uq2t} and Theorem~\ref{thm_nonlin} below. 

The local data result, in the spirit of Theorem~\ref{thm_uq2}, is not so straightforward because the recovery of the Neumann data is not so direct.

Define the ``cone''
\be{Cone}
C_{\hat t, \hat x} :=\{(t,x)\in \R_+\times\bo; \;  t+\dist_0(x,\hat x)\le \hat t\},
\ee
where for $a$, $b$ in $\R^n\setminus \Omega$, $\dist_0(a,b)$ is the infimum of the lengths of all smooth curves lying in $\R^n\setminus \Omega$ that connect $a$ and $b$. 

\begin{theorem}  \label{thm_uq2t}
Let $\Omega$, $\Omega_1$, $\mathcal{G}$ and $\Sigma_s$  be as in Theorem~\ref{thm_uq2}. Let $w$ solve \r{w2} and let $a$, $w$ have even extensions that satisfy \r{reg} and \r{alpha} with $K$ being the closure of the set in \r{th_5} below. Assume that  $w=0$ on $\mathcal{G}$. Then 
\be{th_5}
F =0 \quad \text{in}\quad   \{x\in\Omega\cap (\cup_s \Sigma_s);\; \exists y\in \Gamma, \;     C_{\dist(x,y), y} \subset\mathcal{G}        \}.
\ee
\end{theorem}

\begin{proof}

Outside $\Omega$, $w$ solves the wave equation $Pw=0$ with zero Cauchy data for $t=0$ and zero Dirichlet data on $\mathcal{G}$, see \r{i1}. This does not allow us immediately to conclude that the Neumann data vanishes there, too. On the other hand, by the finite domain of dependence result in \cite{finchPR}, $\partial w/\partial\nu=0$  near $(t,y)\in \mathcal{G}$, if the ``cone'' $C_{t, y} $ is contained in $\mathcal{G}$. 
By Theorem~\ref{thm_uq2}, 
this implies $w(x)=0$ in a neighborhood of all $x$  in the union of all $\Sigma_s$ so that $\dist(x,y) \le t$. This completes the proof. 
\end{proof}

\begin{remark}\label{rem_un} 
In case of observations on the whole boundary, i.e., when $\Gamma=\bo$, with $\mathcal{G}=[0,T]\times \bo$,  \r{th_5} implies $F=0$ in the set $\dist(x,\bo)\le T$. In particular, 
\be{T}
T>\dist(\Omega,\bo)
\ee
is sufficient to conclude $F=0$, see also \r{thm_uq_c3}. This in agreement with the results of the previous section since one can recover the Neumann data easily by Lemma~\ref{lemma_C}. 
\end{remark}

\begin{example}\label{ex_5'}
Let $g$ be Euclidean and let $\Omega$ be strictly convex. Let $\Gamma =\{x^n>a  \}\cap \bo$ with some fixed $a$, and let $\Omega_a =\{x^n>a  \}\cap \Omega$. Then $\Omega_a$ satisfies the foliation condition. For any $x\in \Omega_0$, let $y$ be the point on $\Gamma$ with the same $x'$ coordinates, where $x'=(x^1,\dots,x^{n-1})$. Then $|x-y|<\dist_0(y,\bo\setminus\Gamma)$ because even when $x\in \{x^n=a\}$ (then $|x-y|= y^n-a$ is maximized), the Euclidean distance from $y$ to $\{x^n=a\}$ minimizes the distance from $y$ to that plane with the constraint that we take it outside $\Omega$. Then the ``cone'' $ C_{|x-y|, y} $ is included in $\mathcal{G}$, if we choose $\mathcal{G}= [0,T]\times\Gamma$ with $T=\dist(\Gamma,\bo\setminus\Gamma)$, see \r{dist}. With that choice of $\mathcal{G}$, under the assumptions of the theorem, we get $F=0$ in $\Omega_a$. In other words, the ``cone condition'' in \r{th_5} is satisfied and therefore it is not restrictive in this case. 

By a perturbation argument, if $g$ is close enough to the Euclidean metric, then $F$ would vanish in a bit smaller set than $\Omega_0$. 
\end{example}

\subsection{Uniqueness for the non-linear problem} We now go back to the problem of determining the sound speed $c$ in \r{1A} from $\Lambda f$ with $f$ fixed and known. Clearly, some conditions on $f$ are needed since when $f=0$, for example, we get no information about $c$. 

Based on Theorem~\ref{thm_2t} or Theorem~\ref{thm_uq2t}, we can easily formulate versions for the non-linear problem. We will formulate a consequence of the latter theorem under conditions that guarantee that we can recover $c$ in the whole $\Omega$. 

\begin{theorem}\label{thm_nonlin}
Let $c$ and $\tilde c$ be two smooth positive speeds equal to $1$ outside $\Omega$.  Let $\Sigma_s$ be as in Theorem~\ref{thm_uq2} and satisfy (a) and (b) there with $F := \tilde c^2-c^2$ and with the strictly convexity assumption in (b) fulfilled w.r.t.\ the speed $c$.  Let 
\be{L1}
\tilde \Lambda f = \Lambda f \quad \text{on $[0,T]\times\bo$,\quad  with  $T$ satisfying \r{thm_uq_c2}}.   
\ee
Assume that for some compact $K\subset\bar\Omega$, 
\be{cond_a} 
\supp(\tilde c-c)\subset K, \quad \Delta f\not=0\quad \text{on $K$}.
\ee
Then $\tilde c=c$ in $\cup \Sigma_s$.

If in particular  $\cup \Sigma_s$ is dense in $\bar\Omega$, and $T>\dist(\Omega,\bo)$, then $\tilde c=c$. 
\end{theorem}

Examples of the latter statement include geodesic balls with a fixed center under the assumption that they are all strictly convex. Then $\tilde c=c$ everywhere except in the center, and by continuity, this is true in the center as well.

\begin{proof}
Set $w=\tilde u-u$, where $\tilde u$ and $u$ solve \r{1A} with the speeds $\tilde c$ and $c$, respectively; and the same $f$. Then $w$ solves \r{w2} with $a$ and $F$ as in \r{Fa}. The condition \r{alpha} is then equivalent to $\Delta\tilde u|_{t=0}\not=0$. By \r{1A}, this is equivalent to $\Delta f\not=0$. Then an application of Theorem~\ref{thm_uq2t} completes the proof. 
\end{proof}

\begin{remark}
The condition $\Delta f\not=0$ may look mysterious at first glance. The stability analysis below shows that it is needed for the linearization to be Fredholm. A simplified look at this condition is the following. Let us remove the need for $c$ to be $1$ outside $\Omega$. Then any harmonic function $f$ is also a time independent solution $u=f$ of the wave equation $(\partial_t^2-c^2(x)\Delta)u=0$, regardless of $c$. Then $\Lambda f$ carries no information about $c$ at all. If $f$ is harmonic only on some open set $U$, then $u=f$ (regardless of the speed) in the light cone with base $U$, and then $a=0$ there, see \r{Fa}. Then the kernel of the linearized map 
would be $C^\infty$ for $y\in U$, as it follows from \r{ker}. 
That implies high instability for the linearization, at least. 
\end{remark}

\subsection{Stability} As a general principle, we have stability in Sobolev spaces if we can detect all singularities at $\mathcal{G}$ where we make measurements, see, e.g., \cite{SU-JFA}. Under the assumptions of Theorem~\ref{thm_nonlin}, that would require the following
\be{cond-stab}
\text{$\forall (x,\xi)\in SK$, the geodesic through $(x,\xi)$ hits $\bo$ for some $t$ with $|t|<T$}.
\ee
Here we used the fact that the problem extends in an even way w.r.t.\ $t$. We also identify vectors and covectors by the metric $g$. 

Notice that condition~\r{cond-stab} is stronger than the uniqueness condition \r{T}. The latter requires that from any $x$ there is a signal (a unit speed curve) originating from $x$ reaching $\bo$ up to time $T$, i.e., $\dist(x,\bo)<T$. Condition~\r{cond-stab} requires that from any $x$ and any direction $\xi$ the geodesic through it reaches $\bo$ for time $|t|<T$. The same conditions appear in the analysis of the thermoacoustic problem of recovery of $f$, given $c$ and $\Lambda f$, see \cite{SU-thermo, QSUZ_skull}. Here however, we   assume the foliation condition as well.

Let $R\mathbf{w}$ be the trace of the first component $w$ of $\mathbf{w} :=[w,w_t]$, defined on $[0,T]\times\Omega$, to $[0,T]\times\bo$.

\begin{theorem}\label{thm_st}
Let $w$ solve \r{w2} with $w$, $a$ satisfying the regularity assumptions of Theorem~\ref{thm_2t} , and let $F$ be supported in a compact $K\subset\Omega$, with $F\in L^2(K)$, and let $a(0,\cdot)|_K\not=0$. Let $\Sigma_s$ be as in Theorem~\ref{thm_uq2} , and assume that $\cup\Sigma_s$ is dense in $\Omega$ (the foliation condition). Let $K$ and $T$ satisfy \r{cond-stab} (the stability condition). Then 
\be{3.3}
\|F\|_{L^2(K)}\le C\|w_{tt}\|_{L^2([0,T]\times\bo)}
\ee
with a constant $C$ that remains uniform when the coefficient $a$ stays in a fixed  bounded set in $C^2([0,T];\; C(\bar\Omega))$.
\end{theorem}
\begin{proof}
We use the notation $a(t)=a(t,\cdot)$ below. 
Differentiate
\be{ker}
\mathbf{w} = \int_0^t \mathbf{U}(s)\left[  0,a(t-s)F   \right] \,\d s
\ee
to get
\[
\partial_tRw =   R\mathbf{U}(t)[0,a(0) F] + R\int_0^t \mathbf{U}(s)[0,a'(t-s) F]\,\d s.
\]
Differentiate again and use the identity $a'(0)=0$ and the definition of $\Lambda$ to get
\be{3.1a}
\partial_t^2 Rw =   \Lambda a(0) F + R\int_0^t\mathbf{U}(s)[0,a''(t-s) F]\,\d s  = \Lambda a(0) F + R\int_0^t \mathbf{U}(t-s)[0,a''(s) F]\,\d s.
\ee
Let $\chi\in C^\infty([0,T])$ be such that $\chi=0$ near $T$, and $\chi=1$ on $[0,T_0]$, where $T_0<T$ is such that \r{T} still holds with $T$ replaced by $T_0$. Let $B$ be the back-projection operator defined as follows. Let $v$ solve
\[
\left\{
\begin{array}{rcll}
(\partial_t^2 -c^2\Delta)v &=&0&  \mbox{in $(0,T)\times \Omega$},\\
v|_{t=T}  = \partial_t v|_{t=T}& =&0,\\
v|_{[0,T]\times\bo} &=& h.
\end{array}
\right.               
\]
Then $Bh := v|_{t=0}$. By \cite[Theorem~3]{SU-thermo}, $B\chi \Lambda$ is a  classical \PDO\ of order $0$ with principal symbol
\[
\frac12 \chi(\gamma_{x,\xi}(\tau_+(x,\xi)))+ \frac12 \chi(\gamma_{x,\xi}(\tau_-(x,\xi))),
\]
where $\gamma_{x,\xi}$ is the unit speed geodesic issued from $(x,\xi)$, and $\pm\tau_\pm\ge0$ are the times need for it to hit $\bo$. We recall that we identify vectors and covectors by the metric $g$. Condition~\r{cond-stab} guarantees that the symbol above is elliptic. Let $Q$ be a parametric for it, i.e., $Q B\chi\Lambda=\Id +K_0$ in $\Omega$, with $K$ smoothing. For the purpose of this proof, we only need $K_0\in \Psi^{-1}$, that can be achieved with finite smoothness requirements on $c$. 

Apply $QB\chi$ to \r{3.1a} to get
\be{3.1ab}
QB\chi\partial_t^2 Rw =   (\Id+K_0) a(0) F + QB\chi \int_0^t R\mathbf{U}(t-s)[0,a''(s) F]\,\d s.
\ee
For any $s\in [0,t]$, the function $[0,a''(s) F]$ belongs to the energy space $\mathcal{H}$ and is supported in $\Omega$. Then   $R\mathbf{U}(t-s)$ (a more accurate notation  would be $R\mathbf{U}(\cdot-s)$) maps that function to a function that belongs to $H^1_{(0)}(\R\times \bo)$, where the subscript $(0)$ indicates a support disconnected from $t=0$. This is  explained in \cite{SU-thermo} in the context of thermoacoustic tomography, and the reason is that $R\mathbf{U}(\cdot-s)$ is an FIO of order $0$ with a canonical relation of graph type, see also \cite{Tataru_98}. The dependence on $s$ is continuous, therefore, the integral in \r{3.1a} belongs to that space, as well. By \cite{LasieckaLT}, $B: H^1_{(0)}(\R\times \bo) \to H_D(\Omega)$ and is continuous.  Therefore, the integral term in \r{3.1a} is a compact operator of $F$ in $L^2(\Omega)$, mapping $L^2(\Omega)$ into $H_D(\Omega)$. 

Therefore, since $a(0)|_K\not=0$, 
\be{3.4}
\|F\|_{L^2(K)}\le C \|\partial_t^2Rw\|_{L^2([0,T]\times \bo)} + C\|K_2F\|_{L^2(\Omega)},
\ee
with $K_2 :L^2(K)\to L^2(\Omega)$ compact. 
We used here the fact that $B: L^2_{\textrm{comp}}(\R\times \bo) \to L^2(\Omega)$ that also follows from \cite{LasieckaLT}; or from the property of $B$ to be an FIO of order $0$ with a canonical relation of graph type \cite{SU-thermo}. By \cite[Proposition~5.3.1]{Taylor-book1}, estimate \r{3.4} implies a similar one, with a different $C$, and the last term missing.

The statement about the uniformity of $C$ does not follow directly from the last argument above because there is no control over $C$. Instead, we will perturb estimate \r{3.3}. Notice first that by \r{3.1a}, the map  $C^2([0,T];\; C(\bar\Omega)) \ni a\mapsto w_{tt}|_{[0,T]\times\bo}\in L^2$ is continuous. Then if $a$ and $\tilde a$ are two coefficients and $w$, $\tilde w$ are the corresponding solutions, we have
\[
\begin{split}
\|F\|_{L^2(K)} \le &C \|w_{tt}\|_{L^2([0,T]\times \bo)} \\
                \le &C \|\tilde w_{tt}\|_{L^2([0,T]\times \bo)}  +C \|\tilde w_{tt}-w_{tt}\|_{L^2([0,T]\times \bo)}   \\
                \le &C \|\tilde w_{tt}\|_{L^2([0,T]\times \bo)}  +C\delta \|F\|_{L^2(K)},
\end{split}
\]
where $\delta\ll1$ when $\tilde a$ is close enough to $a$ in $C^2([0,T];\; C(\bar\Omega))$. We can therefore absorb the $\delta$ term with the l.h.s.
\end{proof}

We are ready now to formulate the stability result for the thermoacoustic problem.

\begin{theorem}\label{thm_st2} 
Let $K\subset\Omega$ be a compact set. 
Let $c$, $\tilde c\in C^k$, $f\in H^{k+1}$, $k>n/2$,  be as in Theorem~\ref{thm_nonlin} and let the assumptions of that theorem be satisfied, except for \r{L1}. Then
\be{3.5}
\|\tilde c-c\|_{L^2(K)}\le C\|\partial_{t}^2\delta \Lambda f\|_{L^2([0,T]\times\bo)},
\ee
with $C=C(C_1)$ uniform if $\tilde c$ satisfies $\|\tilde c\|_{C^k}\le C_1$, $1/C_1\le \tilde c$, $k>n/2$. 
\end{theorem}

\begin{proof}
Apply Theorem~\ref{thm_st} with $a$ and $F$ as in \r{Fa}. We only need to prove the statement about the uniformity of $C$. That requires to estimate $\|\tilde u\|_{C^2([0,T]\times\bo )}$ in terms of $\tilde c$, see \r{Fa} and Theorem~\ref{thm_st}.

It is straightforward to see that if $c\in C^{k-1}$, then $\mathbf{A}\!^k\mathbf{f}\in 
\mathcal{H}$ provided that $\mathbf{f}\in H^{k+1}\times H^k$. Then $\mathbf{A}\!^k\mathbf{U}(t)\mathbf{f}$ is locally in $H^{k+1}\times H^k$, therefore the first component of $\mathbf{U}(t)\mathbf{f}$ is $C^{k+1-n/2}$ provided that $k+1-n/2>0$. We have  $k+1-n/2\ge2$ when $k-1\ge n/2$. Therefore, $ \tilde u\in C([0,T];\; C^2)$ when $k>n/2$. The analysis of the rest of the second derivatives of $u$ is similar. 
\end{proof}


\begin{thebibliography}{10}

\bibitem{BardosLR_control}
C.~Bardos, G.~Lebeau, and J.~Rauch.
\newblock Sharp sufficient conditions for the observation, control, and
  stabilization of waves from the boundary.
\newblock {\em SIAM J. Control Optim.}, 30(5):1024--1065, 1992.

\bibitem{Bukh_Kl_81}
A.~L. Bukhgeim and M.~V. Klibanov.
\newblock Uniqueness in the large of a class of multidimensional inverse
  problems.
\newblock {\em Dokl. Akad. Nauk SSSR}, 260(2):269--272, 1981.
\newblock English translation: Soviet Math. Dokl. 24 (1981), no. 2, 244--247
  (1982).

\bibitem{finchPR}
D.~Finch, S.~K. Patch, and Rakesh.
\newblock Determining a function from its mean values over a family of spheres.
\newblock {\em SIAM J. Math. Anal.}, 35(5):1213--1240 (electronic), 2004.

\bibitem{ImanuvilovY01}
O.~Y. Imanuvilov and M.~Yamamoto.
\newblock Global {L}ipschitz stability in an inverse hyperbolic problem by
  interior observations.
\newblock {\em Inverse Problems}, 17(4):717--728, 2001.
\newblock Special issue to celebrate Pierre Sabatier's 65th birthday
  (Montpellier, 2000).

\bibitem{ImanuvilovY01A}
O.~Y. Imanuvilov and M.~Yamamoto.
\newblock Global uniqueness and stability in determining coefficients of wave
  equations.
\newblock {\em Comm. Partial Differential Equations}, 26(7-8):1409--1425, 2001.

\bibitem{ImanuvilovY03}
O.~Y. Imanuvilov and M.~Yamamoto.
\newblock Determination of a coefficient in an acoustic equation with a single
  measurement.
\newblock {\em Inverse Problems}, 19(1):157--171, 2003.

\bibitem{Isakov04}
V.~Isakov.
\newblock Carleman type estimates and their applications.
\newblock In {\em New analytic and geometric methods in inverse problems},
  pages 93--125. Springer, Berlin, 2004.

\bibitem{Isakov-book}
V.~Isakov.
\newblock {\em Inverse problems for partial differential equations}, volume 127
  of {\em Applied Mathematical Sciences}.
\newblock Springer, New York, second edition, 2006.

\bibitem{LasieckaLT}
I.~Lasiecka, J.-L. Lions, and R.~Triggiani.
\newblock Nonhomogeneous boundary value problems for second order hyperbolic
  operators.
\newblock {\em J. Math. Pures Appl. (9)}, 65(2):149--192, 1986.

\bibitem{LasieckaTY}
I.~Lasiecka, R.~Triggiani, and P.~F. Yao.
\newblock Exact controllability for second-order hyperbolic equations with
  variable coefficient-principal part and first-order terms.
\newblock In {\em Proceedings of the {S}econd {W}orld {C}ongress of {N}onlinear
  {A}nalysts, {P}art 1 ({A}thens, 1996)}, volume~30, pages 111--122, 1997.

\bibitem{QSUZ_skull}
J.~Qian, P.~Stefanov, G.~Uhlmann, and H.~Zhao.
\newblock A new numerical algorithm for thermoacoustic and photoacoustic
  tomography with variable sound speed.
\newblock {\em submitted}, 2010.

\bibitem{Romanov06}
V.~G. Romanov.
\newblock Carleman estimates for a second-order hyperbolic equation.
\newblock {\em Sibirsk. Mat. Zh.}, 47(1):169--187, 2006.

\bibitem{SU-JFA}
P.~Stefanov and G.~Uhlmann.
\newblock Stability estimates for the hyperbolic {D}irichlet to {N}eumann map
  in anisotropic media.
\newblock {\em J. Funct. Anal.}, 154(2):330--358, 1998.

\bibitem{SU-JAMS}
P.~Stefanov and G.~Uhlmann.
\newblock Boundary rigidity and stability for generic simple metrics.
\newblock {\em J. Amer. Math. Soc.}, 18(4):975--1003, 2005.

\bibitem{SU-thermo}
P.~Stefanov and G.~Uhlmann.
\newblock Thermoacoustic tomography with variable sound speed.
\newblock {\em Inverse Problems}, 25(7):075011, 16, 2009.

\bibitem{SU-thermo_brain}
P.~Stefanov and G.~Uhlmann.
\newblock Thermoacoustic tomography arising in brain imaging.
\newblock to appear in Inverse Problems, 2010.

\bibitem{Tataru_98}
D.~Tataru.
\newblock On the regularity of boundary traces for the wave equation.
\newblock {\em Ann. Scuola Norm. Sup. Pisa Cl. Sci. (4)}, 26(1):185--206, 1998.

\bibitem{Tataru99}
D.~Tataru.
\newblock Unique continuation for operators with partially analytic
  coefficients.
\newblock {\em J. Math. Pures Appl. (9)}, 78(5):505--521, 1999.

\bibitem{Tataru04}
D.~Tataru.
\newblock Unique continuation problems for partial differential equations.
\newblock In {\em Geometric methods in inverse problems and {PDE} control},
  volume 137 of {\em IMA Vol. Math. Appl.}, pages 239--255. Springer, New York,
  2004.

\bibitem{Taylor-book1}
M.~E. Taylor.
\newblock {\em Partial differential equations. {I}}, volume 115 of {\em Applied
  Mathematical Sciences}.
\newblock Springer-Verlag, New York, 1996.
\newblock Basic theory.

\bibitem{TriggianiY}
R.~Triggiani and P.~F. Yao.
\newblock Carleman estimates with no lower-order terms for general {R}iemann
  wave equations. {G}lobal uniqueness and observability in one shot.
\newblock {\em Appl. Math. Optim.}, 46(2-3):331--375, 2002.
\newblock Special issue dedicated to the memory of Jacques-Louis Lions.

\bibitem{Yao_control1}
P.-F. Yao.
\newblock On the observability inequalities for exact controllability of wave
  equations with variable coefficients.
\newblock {\em SIAM J. Control Optim.}, 37(5):1568--1599 (electronic), 1999.

\end{thebibliography}

\end{document}